\documentclass{amsart}
\usepackage{amsmath,amssymb,hyperref,mathrsfs,graphicx}
\usepackage[utf8x]{inputenc}
\usepackage{amsmath}
\usepackage{amsfonts}
\usepackage{latexsym}
\usepackage{amsthm}
\usepackage{amssymb,amscd}
\usepackage{xargs}
\usepackage{tikz}
\usepackage{graphicx}
\usepackage{color}
\usepackage{enumitem}
\usepackage[colorinlistoftodos,prependcaption]{todonotes}
\presetkeys%
{todonotes}%
{inline,backgroundcolor=white}{}
\tikzset{/tikz/notestyleraw/.append style={text=black}}

\newtheorem{thm}{Theorem}[section]

\newtheorem{lem}[thm]{Lemma}
\newtheorem{defn}[thm]{Definition}
\newtheorem{prop}[thm]{Proposition}
\newtheorem{cor}[thm]{Corollary}

\newtheorem{rmk}[thm]{Remark}
\newcommand{\be}{\begin{eqnarray}}
\newcommand{\ee}{\end{eqnarray}}
\newcommand{\beal}{\begin{aligned}}
\newcommand{\enal}{\end{aligned}}

\newcommand{\eps}{\varepsilon}

\newcommand{\lb}{\lambda}

\newcommand{\T}{\mathbb{T}}
\newcommand{\R}{\mathbb{R}}

\newcommand{\om}{\omega}
\newcommand{\dt}{\delta}

\newcommand{\cM}{\mathcal{M}}

\newcommand{\cF}{\mathcal{F}}
\newcommand{\cT}{\mathcal{T}}
\newcommand{\cA}{\mathcal{A}}

\newcommand{\cL}{\mathcal{L}}

\newcommand{\wh}{\widehat }
\newcommand{\wt}{\widetilde }

\title[Convergence of the viscosity solutions of contact H-J equations]{Convergence of viscosity solutions of generalized contact Hamilton-Jacobi equations}
\thanks{Email: $*$yananwang@fudan.edu.cn,\quad$\dagger$yanjun@fuan.edu.cn,\quad $\ddagger$ jellychung1987@gmail.com}
\date{}
\begin{document}
\maketitle

\centerline{\scshape Ya-Nan Wang$^*$}
\medskip
{\footnotesize
\centerline{School of Mathematical Sciences, Nanjing Normal University, Nanjing, 210097, China}
% please put the address of the first author

}
\bigskip

\centerline{\scshape Jun Yan$^\dagger$}
\medskip
{\footnotesize
\centerline{School of Mathematical Sciences, Fudan University and Shanghai Key Laboratory}
\centerline{ for Contemporary Applied Mathematics, Shanghai 200433, China}
% please put the address of the first author

}
\bigskip

\centerline{\scshape Jianlu Zhang$^\ddagger$}
\medskip
{\footnotesize
% please put the address of the first author
\centerline{Hua Loo-Keng Key Laboratory of Mathematics \&}
 \centerline{Mathematics Institute, Academy of Mathematics and systems science}
 \centerline{Chinese Academy of Sciences, Beijing 100190, China}
}
\bigskip

\begin{abstract}
For any compact connected manifold $M$, we consider the generalized contact Hamiltonian $H(x,p,u)$ defined on $T^*M\times\R$ which is conex in $p$ and monotonically increasing in $u$. Let $u_\epsilon^-:M\rightarrow\R$ be the viscosity solution of the parametrized contact Hamilton-Jacobi equation
\[
H(x,\partial_x u_\epsilon^-(x),\epsilon u_\epsilon^-(x))=c(H)
\]
with $c(H)$ being the Ma\~n\'e Critical Value. We prove that $u_\epsilon^-$ converges uniformly,  as $\eps\rightarrow 0_+$, to a specfic viscosity solution $u_0^-$ of the critical equation
\[
H(x,\partial_x u_0^-(x),0)=c(H)
\]
 which can be characterized as a minimal combination of associated Peierls barrier functions. 
\end{abstract}

\vspace{40pt}

\section{Introduction}\label{s1}
\vspace{20pt}

 For a smooth compact Riemannian manifold $M$ without boundary, the Hamiltonian $H$ is usually characterized as a $C^r-$function ($r\geq2$) on the contagent bundle $T^*M$, with the associated Hamilton's equation
 \be\label{eq:ham}
 \left\{
\begin{aligned}
\dot x&=\partial_p H(x,p)\\
\dot p&=-\partial_x H(x,p)
\end{aligned}
\right.
\ee
for $(x,p)\in T^*M$. The Hamilton's equation describes the movement of macroscopic objects with conservative energy, of which the Hamiltonian $H(x,p)$ works as a {\sf First Integral} of (\ref{eq:ham}). Forecasting the complete movement from any initial point $(x,p)\in T^*M$ became a time-honored target since the age of Newton, until the chaos phenomenon was firstly revealed by Poincar\'e, in his research on the generalized solution of  {\sf Three Body Problem} \cite{P}. That ends the endeavor to give an explicit solution for (\ref{eq:ham}) and leads to a new direction of research: categorizing certain {\sf invariant sets} and verifying their topological stability. This conversion leads to a prosperity of quanlitative theory in the last centrury.\\

For Hamiltonians positive definite of the momentum $p$, e.g. the mechanical Hamiltonian, the Riemannian metric etc, we can translate the Hamiltonian dynamics into Lagrangian dynamics via the {\sf Legendre transformation}, then each trajectory of (\ref{eq:ham}) becomes a critical curve of the {\sf Euler-Lagrange equation} and vice versa. Benefiting from the {\sf Aubry-Mather theory} \cite{Ma} and the {\sf weak KAM solutions} \cite{F}, we now have a clearer understanding of the invariant sets with variational meaning. Besides, variational connecting orbits between these invariant sets are also constructed in \cite{BKZ,CY1,CY2}, as a dynamical extension.\\

Notice that all the motions of the masses would inevitably sustain friction from the environment, e.g. the wind, the fluid, interface etc. That leads to a dissipation of energy. From a viewpoint of PDE, Lions and Barles firstly studied this kind of dissipative equations in \cite{Ba, L}. Following their ideas, now we consider $H(x,p,u)$ on $T^*M\times\R$ which is $C^2$ and satisfies the following standing assumptions:

\begin{description}
	\item [(H1)]\textbf{Positive definiteness} For every $(x,p,u)\in T^*M\times\mathbf{R}$, the second order partial derivative $\partial^2H/\partial p^2(x,p,u)$ is positive definite as a quadratic form.
	\item [(H2)] \textbf{Superlinearity} For every $(x,u)\in M\times\mathbf{R}$, $\lim_{|p|\rightarrow+\infty}H(x,p,u)/|p|=+\infty$.
	\item [(H3)] \textbf{Monotonicity} There exists $\Delta>0$, such that 
	$
	0<\frac{\partial H}{\partial u}(x,p,u)\leq\Delta$ for all $ (x,p,u)\in T^*M\times\mathbf{R}.$
\end{description}
In some literatures, (H1)-(H2) are usually called {\sf Tonelli conditions} \cite{F}. Then for any $(x,p,u)\in T^*M\times\R$, the corresponding {\sf contact Hamilton's equations} can be written by 
\begin{equation}\label{eq:ode}
	\begin{cases}
	\dot x =\frac{\partial H}{\partial p}(x,p,u), \\
\dot p= -\frac{\partial H}{\partial x}(x,p,u)-p\frac{\partial H}{\partial u}(x,p,u), \\
\dot u=p\cdot \frac{\partial H}{\partial p}(x,p,u)-H(x,p,u).	
	\end{cases}
\end{equation}

%
%(H3) enures every trajectory of (\ref{eq:ode}) is complete, and along each trajectory the Hamiltonian $H(x,p,u)$ keeps damping in the forward time (when energy is positive). 
\begin{rmk}
Particularly, if the Hamiltonian is linear of $u$, the previous two equations of (\ref{eq:ode}) are independent of $u-$variable so we can still take $T^*M$ as the phase space and the $3^{rd}$ equation of (\ref{eq:ode}) becomes an additional control. This system is usually called {\sf discounted}. For such a system, the flow $\varphi_{H,\lb}^t$ transports the standard symplectic form into a multiple of itself, i.e.
\[ 
(\varphi_{H,\lb}^t)^*dp\wedge d x= e^{\lb t}dp\wedge d x,\quad\forall \ t\in\R
\]
with $\lb\equiv\partial_uH$ being a constant. In different literatures we may encounter this system but with different names, e.g. {\sf conformally symplectic system} \cite{CCD,MS}, {\sf dissipative system} \cite{WL}, or notably {\sf Duffing equation} when $M$ is an $1-$dimensional manifold \cite{MH}. Physically, this kind of equation describes the mechanical motion of masses with friction proportional to the velocity, which is widely considered in astronomy \cite{CC}, transport \cite{WL} and economics \cite{B}.
\end{rmk}

Following the setting of \cite{Ba}, we consider the {\sf Hamilton-Jacobi equation} satisfying our standing assumptions and parameterize it by a $\epsilon\in[0,1]$, i.e.
\be\label{eq:hj}
H(x,\partial_x u_\epsilon(x),\epsilon u_\epsilon(x))=c(H)
\ee
for a constant $c(H)$ posteriorly decided. By the {\sf Comparison Principle} the viscosity solution $u_\epsilon$ of (\ref{eq:hj}) is unique (see \cite{Ba} or Theorem 3.2 of \cite{CCIZ}), so a natural  question is to study the asymptotic behavior of $u_\epsilon$ as $\epsilon$ tends to zero. 
A heuristic idea is to show that $u_\epsilon$ are equi-Lipschitz and uniformly bounded when $\epsilon\in(0,1]$, if so, by the {\sf Ascoli-Alzel\`a Theorem} each accumulating function of $u_\epsilon$ would be a viscosity solution of the conservative Hamilton-Jacobi equation 
\be\label{eq:hj-0}
H(x,\partial_x u(x),0)=c(H)
\ee
and accordingly $c(H)$ has to be the {\sf Ma\~n\'e Critical Value} (see Appendix \ref{a1} for the definition). By defining the {\sf Tonelli Lagrangian}:
%Moreover, due to the Legendre transformation, the {\sf Tonelli Lagrangian}
\[
L(x,\dot x,u)=\max_{p\in T_x^*M}\{\langle\dot x,p\rangle-H(x,p,u)\},
\]
the uniqueness of the accumulating function of $u_\epsilon$ can be stated as the following:
%supplies us a variational framework to prove the convergence of $u_\epsilon$:
\begin{thm}[Main 1]\label{thm:1}
Let $H:T^*M\times\R \rightarrow\R$ be a Hamiltonian satisfying the standing assumptions. For $\epsilon > 0$, we denote by $u_\epsilon :M\rightarrow\R$ the unique continuous viscosity solution of (\ref{eq:hj}), then the family $u_\epsilon$ converges to a unique viscosity solution $u_0$ of (\ref{eq:hj-0}) as $\eps\rightarrow 0$, which is the largest critical subsolution $u:M\rightarrow\R$ of (\ref{eq:hj-0}) such that for every projected Mather measure $\mu$ (defined in Appendix \ref{a1}), 
\[
\int_{M} u(y)\cdot\frac{\partial L}{\partial u}(y,v(y),0)d\mu(y)\geq0.
\]
\end{thm}
The second conclusion we want to display is a dynamic interpretation of previous $u_0$, by using the Peierls barrier functions (see Appendix \ref{a1} for the definition):
\begin{thm}[Main 2]\label{thm:2}
The limit function $u_0$, obtained in Theorem \ref{thm:1} above, can be characterized in the following way:
%\item it is the largest critical subsolution $u:M\rightarrow\R$ such that for every projected Mather measure $\mu$, 
%\be
%\frac{\int_M u(y)\cdot\frac{\partial L}{\partial u}(y,v(y),0)d\mu}{\int_M\frac{\partial L}{\partial u}(y,v(y),0)d\mu(y)}\leq0;
%\ee
 it is the infimum of the functions $h_\mu^\infty$ defined by 
\be\label{eq:con-1}
h_\mu^\infty(x):=\dfrac{\int_Mh^\infty(y,x)\cdot\frac{\partial L}{\partial u}(y,v(y),0)d\mu(y)}{\int_M\frac{\partial L}{\partial u}(y,v(y),0)d\mu(y)},\quad\forall\ x\in M
\ee
over all projected Mather measures $\mu$, where $h^\infty(\cdot,\cdot)$ is the Peierls barrier of the conservative Hamiltonian $H(x,p,0)$.
\end{thm}

The first result about the vanishing viscosity limit of solutions to the Hamilton-Jacobi equation was achieved by the group of Davini-Itturiaga-Fathi-Zavidovique for the discounted systems \cite{DFIZ}. Later, a similar result was proved for contact Hamilton systems by the group of Chen-Cheng-Ishii-Zhao in \cite{CCIZ}, by proposing a reasonable asymptotic condition of $H$, as an extension of Fathi's idea. It's remarkable that in \cite{CCIZ} only a $C^0-$smoothness of $H$ is needed. However, they required the following (SH5) assumption:\medskip

{\tt There exist positive constants $\kappa_R^\eps\leq K_R^\eps$ depending on $R$ and $\eps$ such that
\[
\eps\kappa_R^\eps|u|\leq |H(x,p,\eps u)-H(x,p,0)|\leq\eps K_R^\eps|u|,\quad\forall\ |p|_x\leq R, |u|\leq R
\]
and there exists a suitably large $R_0>0$ such that
\[
\lim_{\eps\rightarrow 0_+}\frac{\kappa_{R_0}^\eps}{K_{R_0}^\eps}=1.\tag{SH5}
\]
}
For $C^1-$smooth $H(x,p,u)$, this (SH5) assumption is equivalent to 
\[
\partial_u H(x,p,0)=\ \text{constant,}
\]
i.e. $H(x,p,u)$ is linearly asymptotic to $H(x,p,0)$. In the current paper we generalize the result of \cite{CCIZ} by removing (SH5), and prove the convergence for more general contact Hamiltonians. Besides, our approach has more dynamic flavors, which benefits from the work on the Aubry-Mather theory of contact Hamilton's equations developed in \cite{Su,WWY}.

\begin{cor}[Discounted Equation]\label{cor:dis}
For the system satisfying $\partial_uH(x,p,u)\equiv 1$ for all $(x,p,u)\in T^*M\times\R$, the parametrized  viscosity solutions $u_\eps$ converges to a unique $u_0$ of (\ref{eq:hj-0}), which can be characterized in the following ways:
\begin{itemize}
\item it is the largest critical subsolution $u:M\rightarrow\R$ such that for every projected Mather measure $\mu$, there holds
\be
\int_M u(y)d\mu\leq0;
\ee
 \item it is the infimum over all projected Mather measures $\mu$ of the functions
$h_\mu^\infty$ defined by 
\be\label{eq:dis-1}
h_\mu^\infty(x)={\int_Mh^\infty(y,x)d\mu(y)},\quad\forall\ x\in M
\ee
where $h^\infty(y,x)$ is a Peierls barrier function of the conservative Hamiltonian $H(x,p,0)$.\end{itemize}
\end{cor}
\begin{rmk}
 
For typical $L(x,v,\eps u)$ nonlinear of  $u$, the limit of $u_\eps$ is generally different from the limit of viscosity solutions to discounted systems. By comparing (\ref{eq:con-1}) with (\ref{eq:dis-1}) we can perceive this, and in Section \ref{s4}, we construct a concrete example to verify it. 

Moreover, it's still open whether we can get different vanishing viscosity limits of solutions subject to different Aubry classes of (\ref{eq:hj-0}), by taking different $1-$jet function $\frac{\partial L}{\partial u}(x,\dot x,0)$. As is shown in (\ref{eq:con-1}), different $1-$jet functions give different combinations of the Peierls barrier functions, which changes the $\alpha-$limit of part of backward calibrated curves (see (\ref{eq:back-cal}) for the definition). Therefore, exploring the relationship between $\frac{\partial L}{\partial u}(x,\dot x,0)$ and the distribution of backward calibrated curves is a fascinating subject.
\end{rmk}

%Theoretically, by taking Lagrangians with different $1-$jet $\frac{\partial L}{\partial u}(x,v,0)$, we expect to extract different viscous solutions of (\ref{eq:hj-0}) subject to different ergodic components of the Mather measures:
%
%\begin{que}
%For each projected ergodic component $\mu^*$ of Mather measures, can we find a contact Lagrangian $L(x,v,\eps u)$, such that the associated viscous solutions $u_\eps$ converges to a limit $u_0$ of the form 
%\[
%u_0(x)=\dfrac{\int_Mh^\infty(y,x)\cdot\frac{\partial L}{\partial u}(y,v(y),0)d\mu^*(y)}{\int_M\frac{\partial L}{\partial u}(y,v(y),0)d\mu^*(y)}, \quad\forall\ x\in M \;?
%\]
%\end{que}
%
%To be filled...........\\

\subsection{Organization of the article} The paper is organized as follows: First, we exhibit some relevant conclusions about the viscosity solutions of the contact Hamilton-Jacobi equation in Section \ref{s2}, which will be used in the proof of our main theorems. Second, we prove the convergence of the viscosity solutions of (\ref{eq:hj}) in Section \ref{s3}. Finally, in Section \ref{s4}, we present an alternative interpretation of the limit solution $u_0$ by using the language of the Aubry-Mather theory, and explore the dynamic difference between our limit $u_0$ and the discounted limit in \cite{DFIZ}. For the consistency and the readability of the proof, we moved some preliminary works of the Aubry-Mather theory and some longsome technical proofs into Appendix \ref{a1}.\\

\subsection*{Acknowledgements} The first author is supported by National Natural Science Foundation of China (Grant No.11501437) and China Postdoctoral Science Foundation(No.2017M611439). The second author is supported by National Natural Science Foundation of China (Grant No. 11631006 and 11790272) and Shanghai Science and Technology Commission (Grant No. 17XD1400500). 
The third author is supported by the National Natural Science Foundation of China (Grant No. 11901560). All the authors are grateful to Prof. Wei Cheng for helpful discussion about the details. 

\vspace{20pt}

\section{Preliminary: contact Hamiltonian and its variational principle}\label{s2}
\vspace{20pt}

In this section, we provide some conclusions for the contact Hamiltonian satisfying the standing assumptions, which can be interpreted as a contact version of the Aubry-Mather theory. Recall that $M$ is a connected and compact manifold without boundary, so we denote by $x\in M$ the coordinate of the configuration space and $p\in T_x^*M$ the $1-$form of the vector space $T_xM$.
%{\color{blue}
%Let $c\in\mathbb{R}$.
%We put 
%\begin{equation}\label{static HJ}
%	H(x,\partial_xu,u)=c.
%\end{equation}
% and the corresponding evolutional Hamilton-Jacobi equation
% \begin{equation}\label{evolutional HJ}
% 	\partial_tu+H(x,\partial_xu,u)=c.
% \end{equation}
%}
%
The Legendre transformation of $H(x,p,u)$ w.r.t. $p$ gives the Lagrangian
\be
L(x,\dot x,u)&=&\max_{p\in T_x^*M}\{\langle\dot x,p\rangle-H(x,p,u)\}
\ee
of which the maximizer is achieved as $\dot x=\partial_pH(x,p,u)$. Indeed, due to the convexity of $H(x,p,u)$ in $p-$variable, the following 
\[
\cL: T^*M\times \R\rightarrow TM\times\R ,\quad\text{via}\quad (x,p,u)\rightarrow (x, H_p(x,p,u),u)
\]
is a diffeomorphism. Due to the conclusion in \cite{Su}, there exists an implicit  {\sf backward Lax-Oleinik operator} $\cT^{-}_t$ defined by 
\begin{equation}\label{forward semigroup}
	\cT^{-}_t\phi(x)=\inf_{\substack{\gamma\in C^{ac}([0,t],M)\\\gamma(t)=x}}\Big\{\phi(\gamma(0))+\int^t_0L(\gamma(s),\dot{\gamma}(s),\cT^{-}_s\phi(\gamma(s)))+c(H)\ \mbox{d}s\Big\},
\end{equation}
for any $\phi(\cdot)\in C^0(M,\R)$ and $t\geq 0$.

%\begin{equation}\label{backward semigroup}
%	T^{c+}_t\phi(x)=\sup_{\gamma\in\mathcal{A}_1}\{\phi(\gamma(t))-\int^t_0L(\gamma(s),\dot{\gamma}(s),T^{c+}_{t-s}\phi(\gamma(s)))+c\ \mbox{d}s\},
%\end{equation}
%where $\mathcal{A}_1=\{\gamma\in C^{ac}([0,t],M)|\gamma(0)=x\}$.
%Moreover, the infimums in (\ref{forward semigroup}) and (\ref{backward semigroup}) can be achieved in $\mathcal{A}_0$ and $\mathcal{A}_1$, respectively.

%{\color{blue}
%A function $u:M\to \mathbb{R}$ is dominated by $L+c$, denoted by $u\prec L(x,\dot{x},u)+c$, means
%that for each piecewise $C^1$ curve $\gamma:[s_1,s_2]\to M$,
%$$
%u(\gamma(s_2))-u(\gamma(s_1))\leq \int^{s_2}_{s_1}L(\gamma(s),\dot{\gamma}(s),u({\gamma(s)}))+c\mbox{d}s.
%$$
%
%
%We put 
%\begin{equation}\label{target_HJ}
%	H(x,\partial_xu,0)=c.
%\end{equation}
%The Ma\~n\'e critical value $C(H)$ is defined by
%$$
%C(H)=\min\{c\in\mathbb{R}: \mbox{equation (\ref{target_HJ}) admits subsolutions}\}.
%$$
%}
%\textg{
%In the following of the paper, we always assume $C(H)=0$, or we replace $H(x,p,u)$ with $H(x,p,u)-C(H)$. We also replace $T^{c\pm}_t\phi(x)$ with $T^\pm_t\phi(x)$, respectively, when $c=0$.
%}

\begin{lem} \cite{Su}\label{lem:vs-contact}
For contact Hamiltonian $H(x,p,u)$ satisfying the standing assumptions, and any $\phi\in C^0(M,\R)$, $\cT_t^-\phi(x)$ uniformly converges as $t\rightarrow+\infty$, to a viscosity solution $u^-(x)\in Lip(M,\R)$ of
\be\label{eq:solution}
H(x,\partial_x u^-,u^-)=c(H),\quad a.e. \; x\in M.
\ee
Moreover, $u^-(x)$ satisfies 
\begin{itemize}
\item for any $x,y\in M$, $s<t\in\R$ and any piecewise $C^1-$continuous curve $\gamma$ connecting them, we have 
\[
u^-(y)-u^-(x)\leq \int_s^t L(\gamma(\tau),\dot{\gamma}(\tau),u^-(\gamma(\tau)))+c(H)\ \mbox{d}\tau;
\]
\item for any $x\in M$, there exists a {\sf backward calibrated curve} $\gamma_x^-:(-\infty,0]\rightarrow M$ ending with it, such that for all $s< t\leq 0$, 
\[
u^-(\gamma_x^-(t))-u^-(\gamma_x^-(s))=\int_s^tL(\gamma_x^-(\tau),\dot{\gamma}_x^-(\tau),u^-(\gamma_x^-(\tau)))+c(H)\ \mbox{d}\tau.
\]
\end{itemize}
Such a viscosity solution is also called a {\sf weak KAM solution} \cite{F}.
\end{lem}
\begin{lem}[Lemma 4.3 of \cite{WWY}]\label{lem:su}
For any $x\in M$, the viscosity solution $u^-$ of (\ref{eq:solution}) is always differentiable along the interior of the backward calibrated curve $\gamma_x^-$ ending with it, namely, $u^-$ is differentiable on the set $\{\gamma_x^-(t)|t\in(-\infty,0)\}$. 
%Besides, there exists a constant $K>0$ such that $|\dot\gamma_x^-(t)|\leq K$ for all $x\in M$ and $t\in(-\infty,0)$.
\end{lem}
\vspace{5pt}

\subsection{Viscosity solutions of parametrized contact Hamiltonians} Now we parameterize the Hamiltonian by $H(x,p,\epsilon u)$ with $\epsilon\in[0,1]$. For $\eps\in(0,1]$, due to Lemma \ref{lem:vs-contact}, we get a family of viscosity solutions $\{u_\epsilon^-\}$ of (\ref{eq:hj}), which  satisfies:
%corresponding to the same critical value $c(H)$, i.e.
%\be\label{eq:solution-para}
%H(x,\partial_x u_\epsilon^-(x),\epsilon u_\epsilon^-(x))=c(H),\quad a.e.\; x\in M
%\ee
%of which the following conclusions hold:

%By applying Lemma \ref{lem:u+} to the parametrized Hamiltonian, we get a sequence of $\{u_\eps^+\}$

\begin{lem}\label{lem:equi-lip}
$\{u_\epsilon^-\}$ is uniformly bounded, and equi-Lipschitz for all $\epsilon\in(0,1]$ with the Lipschitz constant $\kappa$ depending only on $H(x,p,0)$.
\end{lem}
\begin{proof}
The proof is postponed to Appendix \ref{a3}.
\end{proof}

%\begin{proof}
%This Lemma follows from previous Lemma \ref{lem:com-cal}.
%\end{proof}

\begin{lem}\label{lem:com-cal}
For any $\eps\in(0,1]$ and any $x\in M$, the backward calibrated curve $\gamma_{x,\eps}^-:(-\infty,0]\rightarrow M$ associated with $u_\eps^-$ has a uniformly bounded velocity, i.e. there exists a constant $K>0$ depending only on $H(x,p,0)$, such that 
$$
|\dot\gamma_{x,\eps}^-(t)|\leq K, \quad\forall \ t\in(-\infty,0).
$$ 
\end{lem}

\proof
Due to Lemma \ref{lem:su}, 
$$
\dot{\gamma}_{x,\epsilon}(t)=\frac{\partial H}{\partial p}(\gamma_{x,\epsilon}^-(t),\partial_xu_{\epsilon}^-(\gamma_{x,\epsilon}^-(t)),\epsilon u_{\epsilon}^-(\gamma_{x,\epsilon}^-(t))), \ t\in (-\infty,0).
$$
Note that $u_{\epsilon}^-$ is proved to be uniformly bounded and equi-Lipschitz due to Lemma \ref{lem:equi-lip}, then
$\dot{\gamma}_{x,\epsilon}^-$ is also uniformly bounded.
\endproof

% \begin{cor}
% The parameterized Mather set $\wt\cM(\eps)$ is uniformly compact in $TM\times\R$ for all $\eps\in(0,1]$.
% \end{cor} 
% \begin{proof}
%Since any orbit in $\cM(\eps)$ must be a backward calibrated curve of $u_\eps^-$, Lemma \ref{lem:com-cal} directly indicates this.
%\end{proof}

\vspace{20pt}

\section{Convergence of the viscosity solutions of contact Hamilton-Jacobi equations}\label{s3}
\vspace{20pt}

In this section we prove Theorem \ref{thm:1}, namely, $u_\eps^-$ converges as $\eps\rightarrow 0_+$, to a particular solution $u_0^-$ of (\ref{eq:hj-0}). Due to Lemma \ref{lem:equi-lip} and Lemma \ref{lem:vs-contact}, any accumulating function $u_0^-$ of $u_\eps^-$ as $\eps\rightarrow 0_+$ will be a viscosity solution of (\ref{eq:hj-0}). If we could find a unique dynamic interpretation of $u_0^-$, it has to be unique. For readability we decompose the proof into a list of progressive propositions.

%Recall that $\wt\cM(\eps)$ is a Lipschitz graph over $\cM(\eps)$, so we can projected $\wt{\mathfrak M}(\eps)$ to $M$, then get the {\sf projected Mather measure set} ${\mathfrak M}(\eps):=\pi_M \wt{\mathfrak M}(\eps)$. Due to the {\sf Ergodic Decomposition Theorem}, ${\mathfrak M}(\eps)$ is a convex subset in the space of all the probability measures w.r.t. a weak$^*$ topology, with all the ergodic measures being the {\sf extremal points}. 

\begin{prop}\label{prop:geq}
For any ergodic Mather measure $\wt\mu\in ex(\wt{\mathfrak M})$ (defined in Appendix \ref{a1}), and any accumulating function $u_0^-(x)$ of $\{u_\eps^-\}$ as $\eps\rightarrow 0_+$, there holds
\be
\int_{TM}\frac{\partial L}{\partial u}(x,v,0)\cdot u_0^-(x)d\wt\mu\geq 0.
\ee
\end{prop}

\begin{proof}
By the Birkhoff Ergodic Theorem, for any ergodic Mather measure $\wt\mu$ there exists a generic Euler-Lagrange flow $(\gamma(s),\dot{\gamma}(s)):=\varphi_{L,0}^s(\gamma(0),\dot{\gamma}(0))$ such that 
$$
\lim_{T\to +\infty}\frac{1}{T}\int^T_0L(\gamma(t),\dot\gamma (t),0)\mbox{d}t=\int_{TM} L(x,v,0)\mbox{d}\wt\mu=-c(H).
$$ 
 Let $u_0^-$ be the uniform limit of a sequence $u_{\epsilon_n}^-$ with $\eps_n\rightarrow 0_+$, then 
\begin{align*}
u_{\epsilon_n}^-(\gamma(T))-u^-_{\epsilon_n}(\gamma(0))&\leq\int^T_0L(\gamma(s),\dot{\gamma}(s),\epsilon_n u^-_{\epsilon_n}(\gamma(s)))+c(H)\mbox{d}s \\	
&=\int^T_0L(\gamma(s),\dot{\gamma}(s),0)+c(H)\mbox{d}s\\
&\ \  \ \ +\int^T_0L(\gamma(s),\dot{\gamma}(s),\epsilon_n u^-_{\epsilon_n}(\gamma(s)))-L(\gamma(s),\dot{\gamma}(s),0)\mbox{d}s.
\end{align*}
Note that
\begin{align*}
&\int^T_0L(\gamma(s),\dot{\gamma}(s),\epsilon_n u^-_{\epsilon_n}(\gamma(s)))-L(\gamma(s),\dot{\gamma}(s),0)\mbox{d}s\\
&= \int^T_0\int^1_0\frac{\mbox{d}}{\mbox{d}\tau}L(\gamma(s),\dot{\gamma}(s),\tau \epsilon_n u^-_{\epsilon_n}(\gamma(s)))\mbox{d}\tau\mbox{d}s\\
&= \int^1_0\int^T_0\frac{\partial L}{\partial u}(\gamma(s),\dot{\gamma}(s),\tau \epsilon_n u^-_{\epsilon_n}(\gamma(s))\cdot \epsilon_n u^-_{\epsilon_n}(\gamma(s)) \mbox{d}s\mbox{d}\tau.
\end{align*}
We derive 
\begin{align*}
	\frac{1}{T}\bigg(u^-_{\epsilon_n}(\gamma(T))-u^-_{\epsilon_n}(\gamma(0))\bigg)&\leq \frac{1}{T}\int^T_0L(\gamma(s),\dot{\gamma}(s),0)\mbox{d}s+c(H)\\
	&\ \ \ \ \  \ +\int^1_0\frac{1}{T}\int^T_0\frac{\partial L}{\partial u}(\gamma(s),\dot{\gamma}(s),\tau \epsilon_n u^-_{\epsilon_n}(\gamma(s))\cdot \epsilon_n u^-_{\epsilon_n}(\gamma(s)) \mbox{d}s\mbox{d}\tau.
\end{align*}
By the Birkhoff Ergodic Theorem, 
\be
& &\lim_{T\to+\infty}\frac{1}{T}\int^T_0\frac{\partial L}{\partial u}(\gamma(s),\dot{\gamma}(s),\tau \epsilon_n u^-_{\epsilon_n}(\gamma(s))\cdot \epsilon_n u^-_{\epsilon_n}(\gamma(s)) \mbox{d}s\nonumber\\
&=&\int_{TM}\frac{\partial L}{\partial u}(x,v,\tau \epsilon_n u^-_{\epsilon_n}(x))\cdot \epsilon_n u^-_{\epsilon_n}(x)\mbox{d}\wt\mu.\nonumber
\ee
Hence,
$$
0=\lim_{T\to+\infty}\frac{1}{T\eps_n}\bigg(u^-_{\epsilon_n}(\gamma(T))-u^-_{\epsilon_n}(\gamma(0))\bigg)\leq \int^1_0\big(\int_{TM}\frac{\partial L}{\partial u}(x,v,\tau \epsilon_n u^-_{\epsilon_n}(x))\cdot  u^-_{\epsilon_n}(x)\mbox{d}\wt\mu\big)\mbox{d}\tau.
$$
By the Dominant Convergence Theorem, as $n$ tends to $+\infty$ we derive 
$$
\int_{TM}\frac{\partial L}{\partial u}(x,v,0)u_0^-(x)\mbox{d}\wt\mu\geq 0
$$
then the assertion follows.
\end{proof}

This proposition inspires us to define the following set:
\begin{defn}
Let's denote by $\cF_-$ the set of all {\sf $c(H)-$viscosity subsolution} (see Appendix \ref{a1} for the definition) $u:M\rightarrow \R$ of (\ref{eq:hj-0}) such that 
\be
\int_{TM}\frac{\partial L}{\partial u}(x,v,0)\cdot u(x)d\wt\mu\geq 0,\quad\forall\ \wt\mu\in\wt{\mathfrak M}.
\ee
\end{defn}
\begin{rmk}
Recall that Proposition \ref{prop:geq} indicates that $\cF_-$ is nonempty since any accumulating function $u_0^-$ of $\{u_\eps^-\}$ as $\eps\rightarrow 0_+$ is an element of $\cF_-$.
%; Notice that here we assume a multiplier $\int_{TM}\frac{\partial L}{\partial u}(x,v,0)d\mu$ to ensure the integral is about a probalibity measure.
\end{rmk}
\begin{lem}
The set $\cF_-$ is uniformly bounded from above, i.e.
\[
\sup\{u(x)|\ \forall\ x\in M,\ u\in\cF_-\}<+\infty.
\]
\end{lem}
\begin{proof}
Recall that the set of $c(H)-$viscosity subsolutions of (\ref{eq:hj-0}) is equi-Lipschitz with a Lipschitz constant $\kappa$ (see Proposition \ref{prop:sub-sol} in Appendix \ref{a1}). For any $u\in\cF_-$, we have
\[
\min_{x\in M}u=\frac{\int_{TM}\frac{\partial L}{\partial u}(x,v,0)\cdot \min_{x\in M} ud\wt\mu}{\int_{TM}\frac{\partial L}{\partial u}(x,v,0)d\wt\mu}\leq \frac{\int_{TM}\frac{\partial L}{\partial u}(x,v,0)\cdot u(x)d\wt\mu}{\int_{TM}\frac{\partial L}{\partial u}(x,v,0)d\wt\mu}\leq 0.
\]
Therefore, $\max_{x\in M}u\leq\max u-\min u$. Due to the equi-Lipschitzness and the compactness of $M$, we have $\max u-\min u\leq \kappa\ \text{diam}(M)<+\infty$.
\end{proof}
As $\cF_-$ is now upper bounded, we can define a supreme subsolution by
\be\label{eq:def-1}
u_0^*:=\sup_{u\in \cF_-} u.
\ee
Later we will see that $u_0^*$ is actually a viscosity solution of (\ref{eq:hj-0}) and the unique accumulating function of $\{u_\eps^-\}$ as $\eps\rightarrow 0_+$.

%\textr{
%On the other side, for any sequence $\eps_n\rightarrow 0_+$ such that $u_{\eps_n}^-\rightarrow u_0^-$, the {\sf shadowing minimal measure set} $\mathscr M_0:=\varlimsup_{n\rightarrow+\infty}\mathfrak M(\eps_n)$ is a subset of $\mathfrak M(0)$. Besides, the associated {\sf shadowing Mather set} can be defined by 
%$\wt\cM_0:=\overline{\bigcup_{\mu_0\in\mathscr M_0}supp(\mu_0)}$, which is an invariant subset of $\wt\cM(0)$ and is a graph above the {\sf projected shadowing Mather set} $\cM_0:=\pi_x\wt{\cM}_0$. Based on this setting, we can prove the following:
%}

\begin{prop}\label{prop:leq}
Let $\om$ be any subsolution of (\ref{eq:hj-0}). For
any $x\in M$, we have
\be
u_\eps^-(x)\geq \om(x)+\int_{TM}\om(y)\int_0^1\frac{\partial L}{\partial u}\Big(y,v,\tau\eps u_\eps^-(y)\Big)d\tau d\wt \mu_x^\eps(y,v).
\ee
where $\wt\mu_x^\eps(\cdot,\cdot)$ is a finite measure defined by 
\be\label{eq:asy-mea}
& &\int_{TM}f(y,v) d \wt\mu_x^\eps(y,v)\\
&=&\eps \int_{-\infty}^0 f(\gamma_{x,\eps}^-(t),\dot \gamma_{x,\eps}^-(t) ) \exp\Big(-\eps\int_0^t\int_0^1\frac{\partial L}{\partial u}\Big(\gamma_{x,\eps}^-(s),\dot\gamma_{x,\eps}^-(s),\tau\eps u_\eps^-(\gamma_{x,\eps}^-(s))\Big)d\tau ds\Big)dt\nonumber
\ee
for any $f(\cdot,\cdot)\in C_c(TM,\R)$. 
%\be
%u_0^-(x)\geq \om(x)+\int_{TM}\om(y)\frac{\partial L}{\partial u}(y,v,0)d\wt \mu_x(y,v).
%\ee
%where $\wt \mu_x(\cdot,\cdot)$ is a finite measure defined by 
%\be\label{eq:asy-mea}
%& &\int_{TM}f(x,v) d \wt\mu_x(x,y)\nonumber\\
%&:=&\lim_{\eps\rightarrow 0_+}\eps \int_{-\infty}^0 f(\gamma_{x,\eps}^-(t),\dot \gamma_{x,\eps}^-(t) ) \exp\Big(-\eps\int_0^t\frac{\partial L}{\partial u}(\gamma_{x,\eps}^-(s),\dot \gamma_{x,\eps}^-(s), 0)ds\Big)dt
%\ee
%for any $f(\cdot,\cdot)\in C_c(TM,\R)$. 
\end{prop}

\begin{proof}
Due to Lemma \ref{lem:su}, for any $x\in M$ and any $\eps\in(0,1]$, there exists a backward calibrated curve $\gamma_{x,\eps}^-:(-\infty,0]\rightarrow M$ ending with $x$. Moreover, the viscosity solution $u_\eps^-(\cdot)$ is differentiable along $\gamma_{x,\eps}^-$ for all $t\in(-\infty,0)$, %For brevity we remove the `$x$' from the subscript and `$-$' from the superscript. 
which implies
\be
\frac d{dt} u_\eps^-(\gamma_{x,\eps}^-(t))&=&L\Big(\gamma_{x,\eps}^-(t),\dot\gamma_{x,\eps}^-(t),\eps u_\eps^-(\gamma_{x,\eps}^-(t))\Big)+c(H)\nonumber\\
&=&L(\gamma_{x,\eps}^-(t),\dot\gamma_{x,\eps}^-(t),0)\nonumber\\
& &+\eps \int_0^1u_\eps^-(\gamma_{x,\eps}^-(t))\frac{\partial L}{\partial u}\Big(\gamma_{x,\eps}^-(t),\dot\gamma_{x,\eps}^-(t),\tau\eps u_\eps^-(\gamma_{x,\eps}^-(t))\Big)d\tau+c(H).\nonumber
\ee
%where $0\leq\theta_\eps(t)\leq\eps u_\eps^-(\gamma_\eps(t))$ for all $t\in(-\infty,0)$. 
%This is because $u_\eps^-$ is equi-Lipschitz and the set of all backward semi-static curves is compact in $TM\times\R$. As

Recall that $\om$ is a subsolution of  (\ref{eq:hj-0}), then 
\[
H(x,\partial_x\om(x),0)\leq c(H).
\]
Let $\delta>0$. According to Proposition (\ref{prop:sub-sol}), there exists a smooth function $w_\delta$ such that 
$$
\|\omega-\omega_\delta\|<\delta \mbox{ and } H(x,\partial_x\omega_\delta(x),0)<c(H)+\delta.
$$
Combining previous two inequalities we get 
%\[
%\frac d{dt} u_\eps^-(\gamma_{x,\eps}^-(t))-\eps u_\eps^-(\gamma_{x,\eps}^-(t)) \int_0^1\frac{\partial L}{\partial u}\Big(\gamma_{x,\eps}^-(t),\dot\gamma_{x,\eps}^-(t),\tau\eps u_\eps^-(\gamma_{x,\eps}^-(t))\Big)d\tau\geq \frac d{dt}\om(\gamma_{x,\eps}^-(t)).
%\]
\[
\frac{\mbox{d}}{\mbox{d}t}\omega_\delta(\gamma^-_{x,\epsilon}(t))\leq
\frac{\mbox{d}}{\mbox{d}t}u_{\epsilon}^-(\gamma^-_{x,\epsilon}(t))-\eps u_\eps^-(\gamma_{x,\eps}^-(t))\int^1_0\frac{\partial L}{\partial u}\Big(\gamma^-_{x,\epsilon}(t),\dot\gamma_{x,\eps}^-(t),\tau\eps u_\eps^-(\gamma_{x,\eps}^-(t))\Big)\mbox{d}\tau+\delta.
\]
For brevity we denote by 
\[
\alpha_{x,\eps}(t):=\int_0^1\frac{\partial L}{\partial u}\Big(\gamma_{x,\eps}^-(t),\dot\gamma_{x,\eps}^-(t),\tau\eps u_\eps^-(\gamma_{x,\eps}^-(t))\Big)d\tau.
\]
Then
\[
\frac d{dt}\Big(u_\eps^-(\gamma_{x,\eps}^-(t))e^{-\eps\int_0^t\alpha_{x,\eps}(s)ds}\Big)\geq e^{-\eps\int_0^t\alpha_{x,\eps}(s)ds} \Big(\frac d{dt}\om_\dt(\gamma_{x,\eps}^-(t))-\dt\Big).
\]
%{\color{blue}
%$$
%e^{-\eps\int_0^t\alpha_{x,\eps}(s)ds} \frac d{dt}\om_\delta(\gamma_{x,\eps}^-(t))\leq
%\frac d{dt}\Big(u_\eps^-(\gamma_{x,\eps}^-(t))e^{-\eps\int_0^t\alpha_{x,\eps}(s)ds}\Big)+\delta e^{-\eps\int_0^t\alpha_{x,\eps}(s)ds}.
%$$
%}
Integrating both sides we get 
\be
& &u_\eps^-(\gamma_{x,\eps}^-(t))\exp\Big(-\eps\int_0^t\alpha_{x,\eps}(s)ds\Big)\Bigg|_{-T}^0\nonumber\\
&\geq&\int_{-T}^0\exp\Big(-\eps\int_0^t\alpha_{x,\eps}(s)
ds\Big)\cdot\Big(\frac d{dt}\om_\delta(\gamma_{x,\eps}^-(t))-\delta\Big)dt\nonumber
\ee
which can be transferred into 
\be
& &u_\eps^-(x)-u_\eps^-(\gamma_{x,\eps}^-(-T))\exp\Big(\eps\int_{-T}^0\alpha_{x,\eps}(s)ds\Big)\nonumber\\
&\geq &\om_\delta(x)-\big(\om_\delta(\gamma_{x,\eps}^-(-T))+\delta T\big)\exp\Big(\eps\int_{-T}^0\alpha_{x,\eps}(s)ds\Big)\nonumber\\
& &-\int_{-T}^0\big(\om_\delta(\gamma_{x,\eps}^-(t))-\delta t\big)\frac d{dt}\exp\Big(-\eps\int_0^t\alpha_{x,\eps}(s)ds\Big)dt\nonumber\\
&=& \omega_\delta(x)-(\omega_\delta(\gamma_{x,\epsilon}^-(-T))+\delta T)\exp\bigg(\epsilon\int^0_{-T}\alpha_{x,\epsilon}(s)ds\bigg)\nonumber\\
& &+\epsilon\int^0_{-T}\omega_\delta(\gamma_{x,\epsilon}^-(t)-\delta t)\cdot \alpha_{x,\epsilon}(t)\cdot\exp\bigg(-\epsilon\int^t_0\alpha_{x,\epsilon}(s)ds\bigg)dt.\nonumber
\ee
Taking $\delta\to 0_+$, we derive that
\be\label{eq:main1-k}
& u_\eps^-(x)-u_\eps^-(\gamma_{x,\eps}^-(-T))\exp\Big(\eps\int_{-T}^0\alpha_{x,\eps}(s)ds\Big)\nonumber\\
&\geq \omega(x)-\omega(\gamma_{x,\epsilon}^-(-T))\cdot\exp\bigg(\epsilon\int^0_{-T}\alpha_{x,\epsilon}(s)ds\bigg)\nonumber\\
&+\epsilon\int^0_{-T}\omega(\gamma_{x,\epsilon}^-(t))\cdot \alpha_{x,\epsilon}(t)\cdot\exp\bigg(-\epsilon\int^t_0\alpha_{x,\epsilon}(s)ds\bigg)dt.
\ee
By Lemmas \ref{lem:su} and \ref{lem:equi-lip}, there exists a constant $a>0$ such that
$\alpha_{x,\epsilon}(t)<-a$ for all $t< 0$. Then
$$
\lim_{T\to+\infty}\exp\bigg(\int^0_{-T}\alpha_{x,\epsilon}(s)ds\bigg)=0.
$$
Taking $T\to+\infty$ of (\ref{eq:main1-k}), we derive 
$$
u_{\epsilon}^-(x)\geq\omega(x)+\epsilon\int^0_{-\infty}\omega(\gamma_{x,\epsilon}^-(t))\cdot \alpha_{x,\epsilon}(t)\cdot\exp\bigg(-\epsilon\int^t_0\alpha_{x,\epsilon}(s)ds\bigg)dt
$$
then complete the proof.
%Now we take a sequence $\eps_n\rightarrow 0_+$ such that $u_{\eps_n}^-$ converges to a limit $u_0^-$, then due to Lemma \ref{lem:equi-lip} and Lemma \ref{lem:com-cal}, by taking $n\rightarrow +\infty$ in above inequality we get 
%\[
%u_0^-(x)\geq \om(x)+\int_{TM}\om(y)\frac{\partial L}{\partial u}(y,v,0)d\wt \mu_x(y,v)
%\]
%with $\wt \mu_x$ expressed by (\ref{eq:asy-mea}). 
%\textr{The right side decides a probability measure $\wt\mu_\eps$, i.e.
%\[
%u_\eps^-(x)\geq\om(x)-\int_M \om(x)d\wt\mu_\eps(x)
%\]
%As $\eps\rightarrow 0_+$, $\gamma_\eps$ converges to a backward semi-static curve $\gamma$ of (\ref{eq:hj-0}) and $\theta_\eps\rightarrow 0$. So $\wt\mu_\eps$ converges to an invariant probability measure $\wt\mu$, which should be identified later.} After that, we will finally get 
%\[
%u_0^-(x)\geq \om(x)-\int_M\om(x)d\wt\mu(x)
%\]
%and the assertion follows.
\end{proof}

The following conclusion implies that if $\wt\mu_x^\eps$ weakly converges to $\wt\mu_x$, then $\wt\mu_x$ has to be a rescaled Mather measure. 
%which is not merely implicitly defined by (\ref{eq:asy-mea}): 

\begin{lem}\label{lem:mat-mea}
Any weak limit $\wh \mu_x$ of the normalized measure 
\be
\wh\mu_x^\eps:=\frac{\wt\mu_x^\eps}{\int_{TM}d\wt\mu_x^\eps}
\ee
 is a Mather measure in $\wt{\mathfrak M}$. Accordingly, $\mu_x:=(\pi_M)_*\wh\mu_x$ is a projected Mather measure.
\end{lem}
\begin{proof}
For brevity let's denote by  
\[
\beta_{x,\eps}(t):=\exp\Big(-\eps\int_0^t\alpha_{x,\eps}(s)ds\Big).
\]
%Due to Lemma \ref{lem:com-cal}, $\wt\mu_x^\eps$ are uniformly bounded in the Radon measure space. 
For any sequence $\eps\rightarrow 0_+$ such that $\wh\mu_x^{\eps}$ weakly converges to a $\wh \mu_x$, it suffices to show $\wh\mu_x\in\wt{\mathfrak M}$ by the following two steps (due to Theorem \ref{thm:mane}):\smallskip

First, we show $\wt\mu_x$ is a closed measure, which is equivalent to show that for any $\phi(\cdot)\in C^1(M,\R)$, 
\[
\lim_{\eps\rightarrow 0_+}\eps \int_{-\infty}^0\frac{d}{dt}\phi(\gamma_{x,\eps}^-(t))\beta_{x,\eps}(t)dt=0.
\]
Taking integration by part for the left side, we get
\[
\eps\beta_{x,\eps}(t)\phi(\gamma_{x,\eps}^-(t))\Big|_{-\infty}^0-\eps\int_{-\infty}^0\phi(\gamma_{x,\eps}^-(t)) d\beta_{x,\eps}(t)
\]
which tends to zero as $\eps\rightarrow 0_+$ since there exists a $\dt>0$ such that for any $x\in M$, $\eps\in(0,1]$ and $t\in(-\infty,0]$, $\frac{\partial L}{\partial u}(\gamma_{x,\eps}^-(t),\dot \gamma_{x,\eps}^-(t), u_\eps^-(\gamma_{x,\eps}^-(t)))\leq-\dt<0$.\smallskip

Next, we will show that 
\be\label{eq:math-act}
\lim_{\eps\rightarrow 0_+}\eps \int_{-\infty}^0\Big[L\Big(\gamma_{x,\eps}^-(t),\dot \gamma_{x,\eps}^-(t), \eps u_\eps^-(\gamma_{x,\eps}^-(t))\Big)+c(H)\Big]\beta_{x,\eps}(t)dt=0.
\ee
Recall that for any backward calibrated curve $\gamma_{x,\eps}^-$, Lemma \ref{lem:su} tells us $u_\eps^-$ is differentiable along it for all $t\in(-\infty,0)$. Therefore, we have
\be
\frac d{dt}u_\eps^-(\gamma_{x,\eps}^-(t))&=&L(\gamma_{x,\eps}^-(t),\dot \gamma_{x,\eps}^-(t), \eps u_\eps^-(\gamma_{x,\eps}^-(t)))+ H(\gamma_{x,\eps}^-(t),\partial_x u_\eps^-(\gamma_{x,\eps}^-(t)),\eps u_\eps^-(\gamma_{x,\eps}^-(t)))\nonumber\\
&=&L(\gamma_{x,\eps}^-(t),\dot \gamma_{x,\eps}^-(t), \eps u_\eps^-(\gamma_{x,\eps}^-(t)))+c(H),
\ee
which implies
\[
\lim_{\eps\rightarrow 0_+}\eps \int_{-\infty}^0\frac{d}{dt}u_\eps^-(\gamma_{x,\eps}^-(t))\beta_{x,\eps}(t)dt=0
\]
since $\wh\mu_x$ proves to be closed measure so (\ref{eq:math-act}) is obtained. Due to Theorem \ref{thm:mane}, the normalized measure $\wh \mu_x$ is indeed a Mather measure. The graphic property of $\wh\mu_x$ follows from the graphic property of the Mather set $\wt\cM$, see Theorem \ref{thm:mat-gra}. Due to Definition \ref{defn:pro-mat}, $\mu_x$ is a projected Mather measure.
\end{proof} 
%\begin{rmk}
%This Lemma implies that supp$(\wh\mu_x)$ is a Lipschitz graph over $\pi_M$supp$(\wh\mu_x)$, since it's a subset of $\wt\cM(0)$. 
%\end{rmk}
\medskip

%\begin{prop} \label{prop:leq}
%$\forall\ \mu_0\in\mathscr M_0$, we have
%\be
%\int_{TM}\frac{\partial L}{\partial u}(x,v,0)\cdot u_0^-(x)d\mu_0\leq 0.
%\ee
%\end{prop}
%\begin{proof}
%\end{proof}

%\begin{prop}\label{prop:barrier-1}
%For the accumulating function $u_0^-(x)=\lim_{n\rightarrow +\infty}u_{\eps_n}^-$, we have
%\be\label{eq:vs-conv-1}
%u_0^-(x)=\inf_{y\in\cM_0}\Big(u_0^-(y)+h^\infty(y,x)\Big)
%\ee
%where $h^\infty(\cdot,\cdot)$ is the Peierl's barrier function of the conservative Hamiltonian $H(x,p,0)$, see Appendix \ref{a1} for precise definition.
%\end{prop}
%\begin{proof}
%\end{proof}

\noindent{\it Proof of Theorem \ref{thm:1}:} Now we are ready to prove $u_0^*$ is the unique accumulating function of $\{u_\eps^-\}$ as $\eps\rightarrow 0_+$. First, for any accumulating function $u_0^-$, due to Proposition \ref{prop:geq}, $u_0^-\in\cF_-$. That implies $u_0^-\leq u_0^*$. On the other side, due to Proposition \ref{prop:leq}, if we take $\om\in\cF_-$, then
\be
u_0^-(x)&\geq& \om(x)+\int_{TM}\om(y)\frac{\partial L}{\partial u}(y,v,0)d\wt \mu_x(y,v)\nonumber\\
&=&\om(x)+\int_{TM}\om(y)\frac{\partial L}{\partial u}(y,v,0)d\wh \mu_x(y,v)\cdot \int_{TM}d\wt\mu_x\\
&\geq&\om(x)\nonumber
\ee
since $\wh \mu_x$ is a Mather measure due to Lemma \ref{lem:mat-mea}. As a result, we get $u_0^-\geq \sup_{\om\in\cF_-} \om=u_0^*$. Combining these two parts we get $u_0^-=u_0^*$.\qed

\vspace{20pt}

\section{Peierls barrier's interpretation of the limit solution}\label{s4}
\vspace{20pt}

In this section, we will give a different characterization of $u_0^*$ as an infimum, by using the Peierls barrier $h^\infty(\cdot,\cdot)$, and the projected Mather set $\cM$. In Section \ref{s3} we have proved the convergence of $u_\eps^-$ to $u_0^*$ as $\eps\rightarrow 0_+$. We will show the equivalence of $u_0^*$ to the following 
\be\label{eq:vs-conv-2}
\wh u_0^-(x)=\inf_{\mu\in\mathfrak M}\frac{\int_Mh^\infty(y,x)\cdot\frac{\partial L}{\partial u}(y,v(y),0)d\mu(y)}{\int_M\frac{\partial L}{\partial u}(y,v(y),0)d\mu(y)},\quad\forall\ x\in M.
\ee

\noindent{\it Proof of Theorem \ref{thm:2}:} This proof can be devided into two steps. First, for any $\om\in\cF_-$, we have
\[
\om(x)-\om(y)\leq h^\infty(y,x),\quad\forall x,y\in M.
\]
That implies
\be
\om(x)&=&\frac{\int_M\om(x)\frac{\partial L}{\partial u}(y,v(y),0)d\mu(y)}{\int_M\frac{\partial L}{\partial u}(y,v(y),0)d\mu(y)}\nonumber\\
&\leq&\frac{\int_M\om(y)\frac{\partial L}{\partial u}(y,v(y),0)d\mu(y)}{\int_M\frac{\partial L}{\partial u}(y,v(y),0)d\mu(y)}+\frac{\int_Mh^\infty(y,x)\frac{\partial L}{\partial u}(y,v(y),0)d\mu(y)}{\int_M\frac{\partial L}{\partial u}(y,v(y),0)d\mu(y)}\nonumber\\
&\leq&\frac{\int_Mh^\infty(y,x)\frac{\partial L}{\partial u}(y,v(y),0)d\mu(y)}{\int_M\frac{\partial L}{\partial u}(y,v(y),0)d\mu(y)},\quad\forall \mu\in\mathfrak M.\nonumber
\ee
By taking the infimum over $\mu\in\mathfrak M$ for the right side we get 
\[
\om(x)\leq \wh u_0^-(x),\quad\forall x\in M, 
\]
then by taking the supremum over $\om\in\cF_-$ for the left side we get $u_0^*\leq \wh u_0^-$. \\

To show $\wh u_0^-\leq u_0^*$, we first show that $\wh u_0^-$ is a viscosity subsolution of (\ref{eq:hj-0}) (see Appendix \ref{a1} for the definition). Due to the Ergodic Decomposition Theorem, any $\mu\in\mathfrak M$ can be represented as a convex combination of a family of ergodic measure $\mu_i\in ex(\mathfrak M)$. On the other side, for any fixed $y\in M$, $h^\infty(y,\cdot)$ is a weak KAM solution of (\ref{eq:hj-0}), which is definitely a subsolution (shown in Proposition \ref{prop:sol}). By the convexity of $H(x,p,0)$ in $p-$variable and the equi-Lipschitz continuity of all the subsolutions (see Proposition \ref{prop:sub-sol}), it follows that 
\[
h_\mu^\infty(x):=\dfrac{\int_Mh^\infty(y,x)\cdot\frac{\partial L}{\partial u}(y,v(y),0)d\mu(y)}{\int_M\frac{\partial L}{\partial u}(y,v(y),0)d\mu(y)},\quad\forall\ x\in M
\]
is a subsolution as well. Once again by Proposition \ref{prop:sub-sol}, we infer that an infimum of $h_\mu^\infty$ over $\mu\in\mathfrak M$ is still a subsolution. So $\wh u_0^-$ is a subsolution of (\ref{eq:hj-0}). By Proposition \ref{prop:sol}, we just need to show $\wh u_0^-\leq u_0^*$ on the projected Aubry set $\cA$ (see the definition in Appendix \ref{a1}), then $\wh u_0^-(x)\leq u_0^*(x)$ for all $x\in M$.
Recall that for any $y\in\cA$ fixed, the function $-h^\infty(\cdot,y)$ is a viscosity subsolution of (\ref{eq:hj-0}), then the following defined
\be
\om(x)&:=&-h^\infty(x,y)+\inf_{\mu\in\mathfrak M}\frac{\int_Mh^\infty(z,y)\frac{\partial L}{\partial u}(z,v(z),0)d\mu(z)}{\int_M\frac{\partial L}{\partial u}(z,v(z),0)d\mu(z)}\nonumber\\
&=&-h^\infty(x,y)+\wh u_0^-(y)\nonumber
\ee
is a subsolution as well. For any $\mu\in\mathfrak M$, we have 
\be
\int_M \om(x)\frac{\partial L}{\partial u}(x,v(x),0)d\mu(x)&=&-\int_M h^\infty(x,y) \frac{\partial L}{\partial u}(x,v(x),0)d\mu(x)\nonumber\\
& &+\wh u_0^-(y) \int_M \frac{\partial L}{\partial u}(x,v(x),0)d\mu(x)\geq 0\nonumber
\ee
due to $\frac{\partial L}{\partial u}(x,v,0)<0$ for all points in $TM\times\R$ and the definition of $\wh u_0^-$. That implies $\om\in\cF_-$. Therefore, $\om(x)\leq u_0^*(x)$ for all $x\in M$. In particular, we have 
\[
u_0^*(y)\geq \om(y)=-h^\infty(y,y)+\wh u_0^-(y)=\wh u_0^-(y),\quad\forall y\in \cA.
\]
So we finish the proof.\qed

%Recall that $\cM_0\subset \cM(0)$, due to (\ref{eq:vs-conv-1}), for any $x\in M$ fixed, there always exists one point $y_x\in\cM_0$ (may not unique),  such that 
% \[
% u_0^-(x)=u_0^-(y_x)+h^\infty(y_x,x)
% \]
 
 \vspace{10pt}
 
 \subsection{A comparison with the discounted system} In this part we make a comparison of the
 vanishing viscosity limit of solutions between the discounted Hamiltonians and the general contact Hamiltonians. For two different Hamiltonians satisfying our standing assumptions, 
 \[
\text{(discounted)}\quad  F(x,p,\eps u)=H_0(x,p)+\eps u
 \]
 and 
 \[
 \text{(contact)}\quad G(x,p,\eps u)=H_0(x,p)+\eps u H_1(x,p)+\eps^2 H_2(x,p,\eps u), 
 \]
 the convergence of associated viscosity solutions $u_{F,\eps}^-$ (resp. $u_{G,\eps}^-$) will be quite different as $\eps\rightarrow 0_+$, once $H_1(x,p)$ doesn't equal to a constant. We will explain this point by the following example.\\
 
\noindent{\bf Example:} Suppose $(x,p)\in T^*\T$ and we take 
\be
H_0(x,p)=\frac12p\big(p+2V(x)\big).
\ee
Notice that for this system there exists a unique periodic orbit $x(t)\in \T$ with 
\[
\dot x(t)=V(x(t)),\quad\forall t\in\R, 
\]
once $V$ keeps positive but not constant.

Due to Corollary \ref{cor:dis}, the discounted vanishing of $F(x,p,\eps u)$ gives us the limit by 
\be
u_F^-(x)&=&\frac 1{T_p}\int_0^{T_p}h^\infty(x(t),x)dt\nonumber\\
&=&\frac1{2\pi}\int_0^{2\pi}h^\infty(\theta,x)\frac{2\pi d\theta}{T_p\cdot V(\theta)}
\ee
 where $T_p$ is the period of $x(t)$. Actually, $\frac{2\pi}{T_p V(\theta)}$ is a density function on $\T$.
%  and
% \be
% u_{2}^-(x_2)=\left\{
% \begin{aligned}
% \int_0^x\sqrt{2(1-\cos\tau)}d\tau, \quad\forall x\in[0,\pi],\\
% \int_0^{2\pi-x}\sqrt{2(1-\cos\tau)}d\tau, \quad\forall x\in(\pi,2\pi).
% \end{aligned}
% \right.
% \ee
 However, if we take $H_1(x,p)=V(x)$, the contact vanishing of $G(x,p,\eps u)$ give us the limit by
 \be
 u_G^-(x)&=&\frac{\int_0^{T_p}h^\infty(x(t),x)f(x(t))dt}{\int_0^{T_p}f(x(t))dt}\nonumber\\
 &=&\frac1{2\pi}\int_0^{2\pi}h^\infty(\theta,x)d\theta.
 \ee
 We have much freedom to choose $V(\cdot)$ such that $ u_F^-(x)\neq u_G^-(x)$.

%\begin{figure}[htb]
%\begin{center}
%\def\svgwidth{12cm}
%\input{#.eps_tex}
%\caption{}
%\label{shadow}
%\end{center}
%\end{figure}

%\todo{Is this note white?}

\vspace{20pt}

\appendix

\section{weak KAM Theorem for autonomous Tonelli Lagrangians}\label{a1}
\vspace{20pt}

For the Tonelli Lagrangian $L(x,v)$, the {\sf critical curve} is usually defined by a $C^1$ curve $\gamma :\R\rightarrow M$, such that the following {\sf Euler-Lagrange equation} holds
 \be\label{eq:el}
 \frac {d}{dt}L_v(\gamma,\dot\gamma)=L_x(\gamma,\dot\gamma)
 \ee
for all $t\in\R$. As a conjugation, the Legendre transformation give us a Hamiltonian
\[
H(x,p):=\max_{v\in T_xM}\langle p,v\rangle-L(x,v)
\]
and each solution of the Euler-Lagrange equation will correspond to a trajectory of (\ref{eq:ham}).
\begin{lem}[Young Inequality] For any $x\in M$, $p\in T_x^*M$ and $v\in T_xM$, 
\[
H(x,p)+L(x,v)\geq \langle p,v\rangle
\]
always holds and the equality is achieved as $v= H_p(x,p)$, or equivalently $p=L_v(x,v)$.
\end{lem}

Notice that the minimizer of the following 
\be\label{eq:act-zero}
h^t(x,y)&:=&\min_{\substack{\gamma\in C^{ac}([0,t],M)\\\gamma(0)=x,\gamma(t)=y}}\int_0^{t}L(\gamma,\dot\gamma)d\tau
\ee
has to be a solution of (\ref{eq:el}) on $[0,t]\subset\R$. On the other side, for any $(x,v)\in TM$, there exists a unique critical curve $\gamma$  starting from it, which can be extended for all $t\in\R$. If we denote by $\varphi_L^t$ the Euler-Lagrange flow, we can make use of the {\sf Birkhoff Ergodic Theorem} and get a $\varphi_L^t-$invariant probability measure $\tilde{\mu}_\gamma$ by
\be 
\quad\quad\int_{TM}fd\tilde{\mu}_\gamma:=\lim_{T\rightarrow+\infty}\frac1T\int_0^Tf(\gamma,\dot\gamma)dt,\quad\forall f\in C_c(TM,\R).
\ee
Gather all these invariant probability measures into a set $\wt{\mathfrak M}_{inv}$, then the {\sf Ma\~n\'e Critical Value} can be defined by 
%for any closed $1-$form $\eta(x)dx$ with $[\eta]=c\in H^1(M,\R)$, the parametrized Lagrangian 
%then we can define the {\sf $c-$action} by 
%\be
%A_c(\mu):=\int L-cd\mu,\quad \mu\in \mathfrak M_L
%\ee
%where $c\in H^1(M,\R)$ is a closed $1-$form. Notice that
% \[
% L_c(x,v):=L(x,v)-\langle c,v\rangle
% \]
%  possesses the same Euler-Lagrange equation with $L(x,v,t)$, so the following {\sf Mather's Alpha function} $\alpha:H^1(M,\R)\rightarrow\R$
\be\label{eq:alpha}
c(H)=-\min_{\tilde{\mu}\in\wt{\mathfrak{M}}_{inv}}\int Ld\tilde{\mu}.
\ee
 The minimizers of (\ref{eq:alpha}) form the {\sf Mather measure set} $\wt{\mathfrak M}$, which is contained in $\wt{\mathfrak M}_{inv}$. The {\sf Mather set} is defined by 
\[
\wt\cM:=\overline{\bigcup_{\tilde{\mu}\in\wt{\mathfrak M}}supp(\tilde{\mu})}\subset TM
\]
which is a closed invariant set.

\begin{thm}[Graphic \cite{Ma}]\label{thm:mat-gra}
$\wt\cM$ is a Lipschitz graph over the {\sf projected Mather set} $\cM:=\pi_M\wt\cM$, i.e.
\[
\pi_x^{-1}:\cM\rightarrow TM
\]
is a Lipschitz function.
\end{thm}
\begin{defn}\label{defn:pro-mat}
Due to Theorem \ref{thm:mat-gra}, for each $\wt\mu\in\widetilde{\mathfrak M}$, there exists a {\sf projected Mather measure} $\mu=(\pi_M)_*\wt\mu$ defined by 
\[
\int_Mf(x)d\mu(x)=\int_{TM}f\circ\pi_{M}(x,v)d\wt\mu(x,v),\quad\forall f\in C^0(M,\R).
\]
We denote by $\mathfrak{M}$ the {\sf projected Mather measure set}.
\end{defn}
%As the conjugation of $\alpha(c)$, we can define the {\sf Mather's Beta function} $\beta: H_1(M,\R)\rightarrow\R$ by
%\be
%\beta(h)=\inf_{\substack{\mu\in\mathfrak{M}_L, \\\rho(\mu)=h}}\int L\;d\mu
%\ee
%where $\rho(\mu)\in H_1(M,\R)$ is defined by 
%\[
%\langle [\lambda], \rho(\mu)\rangle:=\int\lambda\;d\mu,\quad\forall \text{\;closed 1-form \;}\lambda \text{\;on\;} M.
%\]
%Due to the positive definiteness assumption, both $\alpha(c)$ and $\beta(h)$ are convex and superlinear. Besides, 
%\be\label{eq:ineq-a-b}
%\langle c,h\rangle\leq\alpha(c)+\beta(h),\quad\forall c\in H^1(M,\R),\; h\in H_1(M,\R),
%\ee
%of which the equality holds only for $c\in D^-\beta(h)$ and $h\in D^-\alpha(c)$, namely $c$ is contained in the {\sf sub derivative set} of $\beta(h)$ and $h$ is contained in the sub derivative set of $\alpha(c)$.\\

Following the notions of \cite{B_0}, we define 
%for any $\gamma\in C^{ac}([t,t'],M)$ 
%\begin{equation}
%A^{0,t}(\gamma)=\int_0^{t}L(\gamma(s),\dot{\gamma}(s),s) ds+c(H)t
%\end{equation}
%and 
the {\sf $c(H)-$action function} by
\begin{equation}
h_{c}^t(x,y)=\inf_{\substack{\xi\in C^{ac}([0,t],M)\\
\xi(0)=x\\
\xi(t)=y}}\int_0^{t}L(\xi(s),\dot{\xi}(s)) ds+c(H)t.
\end{equation}
%Therefore, the {\sf Ma\~n\'e Potential function}
%\begin{equation}
%F_c(x,y)=\inf_{t\in[0,+\infty)}h_c^t(x,y),
%\end{equation}
%is well defined on $M\times M$. Furthermore,
and the {\sf Peierls barrier function} by
\be
h^\infty(x,y)=\liminf_{t\rightarrow+\infty}h_c^t(x,y).
\ee
The {\sf projected Aubry set} is defined by 
\[
\cA=\{x\in M|h^\infty(x,x)=0\}.
\]
We can see that $\cA$ is a closed set of $M$ and contains $\cM$. Besides, it can be further decomposed due to the following:

\begin{defn}
The {\sf projected Aubry class} is defined by the element in $\cA/d_c$ where the metric $d_c:\cA\times \cA\rightarrow\R$ is defined by:
\be\label{eq:aubry-clas}
 d_c(x,y)=h^\infty(x,y)+h^\infty(y,x).
\ee
\end{defn}

\begin{defn}
 A function $u:M\rightarrow\R$ is called a {\sf viscosity subsolution}, or {\sf subsolution} for short ($u\prec L+c(H)$), if $u(y)-u(x)\leq h_c^t(x,y)$ for all $(x,y)\in M\times M$ and $t\geq 0$. 
%A curve $\gamma:(-\infty,s)\rightarrow M$ is called {\sf backwrad calibrated}, if 
%\be\label{eq:back-cal}
%u(\gamma(b))-u(\gamma(a))=A_c^{a,b}(\gamma)
%\ee
%for all $a<b\in\R$.
\end{defn}
The following properties have been proved in \cite{F,FS_0,FS}:
\begin{prop}\label{prop:sub-sol}
For a Tonelli Hamiltonian $H$, 
\begin{itemize}
\item all the subsolutions are equi-Lipschitz, we have 
\[
H(x,\nabla u(x))\leq c(H),\ \ a.e.\ x\in M;
\]
\item each subsolution $u$ is a real solution on $\cA$ and keeps differentiable;
\item for any $y\in M$ fixed, $-h^\infty(\cdot,y)$ is a subsolution;
\item if $u$ is the pointwise supremum of a family of subsolutions, then $u$ is a subsolution;
\item if $u$ is the pointwise infimum of a family of subsolutions, then $u$ is a subsolution;
\item if $u$ is a convex combination of a family of subsolutions, then $u$ is a  subsolution;
\item for any subsolution $u$ and any $\eps>0$, there exists a $C^\infty$ function $u_\eps$ such that  $\|u_\eps-u\|\leq\eps$ and $H(x,\partial_x u_\eps(x))\leq c(H)+\eps$.
\end{itemize}
\end{prop}
\begin{defn}
A function  $u^-:M\rightarrow\R$ is called a {\sf weak KAM solution (or viscosity solution} of a Tonelli Hamiltonian $H(x,p)$, if 
\begin{itemize}
\item $u^-\prec L+c(H)$;
\item $\forall x\in M$, there exists a {\sf backward calibrated curve} $\gamma_x^-:(-\infty,0]\rightarrow M$ of $u^-(\cdot)$ ending with it. Namely, 
\be\label{eq:back-cal}
u^-(\gamma_x^-(b))-u^-(\gamma_x^-(a))=h_c^{b-a}(\gamma_x^-(a),\gamma_x^-(b)).
\ee
for all $a<b\leq 0$.
\end{itemize}
\end{defn}
Here we exhibit a list of properties the weak KAM solutions possess, which 
can be found in
\cite{F, FS}.
\begin{prop}\label{prop:sol}
\begin{itemize}
\item For any $y\in M$ fixed, $h^\infty(y,\cdot)$ is a viscosity solution.
\item The viscosity solution $u^-$ is differentiable along each backward calibrated curve $\gamma_x^-$ for all $t\in(-\infty,0)$.
\item Suppose $\om, u^-$ is a subsolution and a viscosity solution respectively. If $\om\leq u^-$ on  $\cA$, then $\om\leq u^-$ on $M$. In particular, if two solutions coincide on $\cA$, then they coincide on the whole manifold $M$.
%\item For each $c \in H^1(M,\R)$, $u_c^-$ is $K_c-$Lipshitz on $M\times\T$.
%\item $u_c^-$ is differentiable at $\mathcal N(c)$;
%\item for any $x\in M$, the function 
%\[
%h_c^\infty(x,\cdot):M\rightarrow\R
%\]
% is a weak KAM solution.
\end{itemize}
\end{prop}
\vspace{20pt}

\section{Ma\~n\'e's Variational Principle of minimal measures}\label{a2}

\vspace{20pt}

In \cite{M}, Ma\~n\'e proposed his own aspect of getting the Mather measure. His approach doesn't need a flow-invariant assumption on the variational space of probability measures, therefore, it is more convenient comparing to Mather's approach.\\

Let $X$ be a metric separable space. A probability measure on $X$ is a nonnegative,
countably additive set function $\mu$ defined on the $\sigma-$algebra $\mathscr B(X)$ of Borel
subsets of $X$ such that $\mu(X) = 1$. In this paper, $X=M$ or its tangent bundle
$TM$. A measure on $TM$ is denoted by $\wt\mu$, and we remove the tilde if we project it to $M$. We say that a sequence
$\{\wt\mu_n \}_n$ of probability measures on TM (weakly) converges to a probability 
measure $\mu$ on $TM$ if
\[
\lim_{n\rightarrow+\infty}\int_{TM}f(x,v)d\wt\mu_n(x,v)=\int_{TM} f(x,v)d\wt\mu(x,v)
\]
for any $f\in C_c(TM,\R)$. Accordingly, 
\be
\lim_{n\rightarrow+\infty}\int_M f(x)d\mu_n(x)&:=&\lim_{n\rightarrow+\infty}\int_{TM} f\circ \pi_M (x,v) d\wt\mu_n(x,v)\nonumber\\
&=&\int_{TM} f\circ\pi_M(x,v)d\wt\mu(x,v)=:\int_M f(x)d\mu(x)\nonumber
\ee
for any $f\in C(M,\R)$.

\begin{defn}
A probability measure $\wh\mu$ on $TM$ is called {\sf closed} if it satisfies:
\begin{itemize}
\item $\int_{TM}|v|d\wh\mu(x,v)<+\infty$;
\item $\int_{TM}\langle \nabla\phi(x),v\rangle d\wh\mu(x,v)=0$ for every $\phi\in C^1(M,\R)$.
\end{itemize}
\end{defn}
Let's denote by $\widetilde{\mathfrak M}_c$ the set of all closed measures on $TM$, then the following conclusion is proved in \cite{M}:
\begin{thm}[Proposition 1.3 of \cite{M}]\label{thm:mane}
$\min_{\wt\mu\in\widetilde{\mathfrak M}_c}\int_{TM}L(x,v)d\widetilde{\mu}(x,v)=-c(H)$. Moreover, the minimizer $\wh\mu$ must be a Mather measure, i.e. $\wh\mu$ is invariant with respect to the Euler-Lagrange flow. 
\end{thm}
\vspace{20pt}

\section{the proof of Lemma \ref{lem:equi-lip}}\label{a3}
\vspace{20pt}

For any $\epsilon\in(0,1]$ and $\phi\in C^0(M,\mathbb{R})$, we have 
$$
\cT^{\epsilon-}_t\phi(x)=\inf_{\gamma(t)=x}\int^t_0L(\gamma(s),\dot{\gamma}(s),\epsilon\cT^{\epsilon-}_s\phi(\gamma(s)))+c(H)ds
$$
implicitly defined, 
with the infimum taken among all piecewise $C^1$ curve. Besides, for each $\phi\in C(M,\mathbb{R})$,
$$
\lim_{t\to+\infty}\cT^{\epsilon-}_t\phi(x)=u^-_{\epsilon}(x),
$$
where $u_\epsilon^-$ is the unique weak KAM solution of $H(x,\partial_xu_{\epsilon}^-,\epsilon u_{\epsilon}^-)=c(H)$, see Theorem 1.4 in \cite{Su} and Proposition A.1 in \cite{WWY}. Now we are ready to give a proof of Lemma \ref{lem:equi-lip}.

\begin{lem} For any  $\phi\in C(M,\mathbb{R})$ with $\|\phi\|\leq 1$,
the family $\{\cT^{\epsilon-}_t\phi(\cdot)|\epsilon\in(0,1],t\geq1\}$ is uniformly bounded.
\end{lem}
\proof
We claim $\{\cT^{\epsilon-}_t\phi(\cdot)|\epsilon\in(0,1],t\geq1\}$ is uniformly bounded from below.
Without loss of generality, we assume $\cT^{\eps -}_t\phi(x)<0$ for some $\epsilon\in (0,1]$ and $(x,t)\in M\times[1,+\infty)$, otherwise $0$ would be a lower bound of the family. Let $\gamma_{x,\epsilon}:[0,t]\to M$ be the associated minimizer of $\cT^{\epsilon-}_t\phi(x)$. Then, there are two probabilities:

\textbf{case I:} There exists $s_0\in[0,t)$ such that $\cT^{\epsilon-}_{s_0}\phi(\gamma_{x,\epsilon}(s_0))=0$ and $\cT^{\epsilon-}_{s}\phi(\gamma_{x,\epsilon}(s))<0, s\in(s_0,t]$.

\textbf{case II:} $\cT^{\epsilon-}_s\phi(\gamma_{x,\epsilon}(s))<0$ for all $s\in[0,t]$.

For case I, due to (H3), $L$ is strictly decreasing of $u$. Therefore,
\begin{align*}
	\cT^{\epsilon-}_t\phi(x)&=\cT^{\epsilon-}_{s_0}\phi(\gamma_{x,\epsilon}(s_0))+\int^t_{s_0}L(\gamma_{x,\epsilon}(s),\dot{\gamma}_{x,\epsilon}(s),\epsilon\cT^{\epsilon-}_s\phi(\gamma_{x,\epsilon}(s)))+c(H)ds\\
	&\geq \int^t_{s_0}L(\gamma_{x,\epsilon}(s),\dot{\gamma}_{x,\epsilon}(s),0)+c(H)ds\\
	&\geq h^{t-s_0}(\gamma_{x,\epsilon}(s_0),x)+c(H)(t-s_0). 
\end{align*}

It is well known that $h^{t-s_0}(\gamma_{x,\epsilon}(s_0),x)+c(H)(t-s_0)$ is uniformly bounded from below (see Lemma 5.3.2 in \cite{F} for instance). Moreover, the lower bound can be made independent of the selection of $\epsilon\in (0,1]$ and $(x,t-s_0)\in M\times (0,+\infty)$. 

For case II, 
\begin{align*}
	\cT^{\epsilon-}_t\phi(x)&=\phi(\gamma_{x,\epsilon}(0))+\int^t_{0}L(\gamma_{x,\epsilon}(s),\dot{\gamma}_{x,\epsilon}(s),\cT^{\epsilon-}_s\phi(\gamma_{x,\epsilon}(s)))+c(H)ds\\
	&\geq \min_{x\in M}\phi(x)+\int^t_{0}L(\gamma_{x,\epsilon}(s),\dot{\gamma}_{x,\epsilon}(s),0)+c(H)ds\\
	&\geq -1+h^t(\gamma_{x,\epsilon}(0),x)+c(H)t.
\end{align*}
Hence, $\{\cT^{\epsilon-}_t\phi(\cdot)|\epsilon\in(0,1]\}$ is bounded from below. Still the lower bound is independent of the selection of $\epsilon\in(0,1]$ and $(x,t)\in M\times [1,+\infty)$. 

As a summary, the family $\{\cT^{\epsilon-}_t\phi(\cdot)|\epsilon\in(0,1],t\geq1\}$ is uniformly bounded from below.\\

We claim $\{\cT^{\epsilon-}_t\phi(\cdot)|\epsilon\in(0,1],t\geq 1\}$ is uniformly bounded from above.
Without loss of generality, we assume $\cT^{\epsilon-}_t\phi(x)>0$ for some $\epsilon\in(0,1]$ and $(x,t)\in M\times [1,+\infty)$, otherwise $0$ is a upper bound of $\{\cT^{\epsilon-}_t\phi(\cdot)|\epsilon\in(0,1],t\geq 1\}$. Let $\beta:[0,t]\to M$ be the associated minimizer of $h^t(\gamma_{x,\epsilon}(0),x)$, i.e.,
$$
h^t(\gamma_{x,\epsilon}(0),x)=\int^t_0L(\beta(s),\dot{\beta}(s),0)ds.
$$
There are also two probabilities:

\textbf{case I'}
 $\cT^{\epsilon-}_s\phi(\beta(s))>0$ for each $s\in[0,t]$.
Hence,
\begin{align*}
	\cT^{\epsilon-}_t\phi(x)&\leq\phi(\beta(0))+\int^t_0L(\beta(s),\dot{\beta}(s),\eps\cT^{\epsilon-}_s\phi(\beta(s)))+c(H)ds\\
	&\leq \max_{x\in M}\phi(x)+\int^t_0L(\beta(s),\dot{\beta}(s),0)+c(H)ds\\
	&\leq 1+h^t(\gamma_{x,\epsilon}(t),x)+c(H)t.
\end{align*}

Since $t\geq 1$, $h^t(\gamma_{x,\epsilon}(t),x)+c(H)t$ is bounded from above. Hence, $\{\cT^{\epsilon-}_t\phi(\cdot)|\epsilon\in(0,t],t\geq1\}$ is uniformly bounded from above.

\textbf{case II'}
There exists $s_1\in [0,t)$ such that $\cT^{\epsilon-}_{s_1}\phi(\beta(s_1))=0$ and $\cT^{\epsilon-}_s\phi(\beta(s))>0,s\in(s_1,t]$. Then,
\begin{align*}
	\cT^{\epsilon-}_t\phi(x)&\leq\cT^{\epsilon-}_{s_1}\phi(\beta(s_1))+\int^t_{s_1}L(\beta(s),\dot{\beta}(s),\eps\cT^{\epsilon-}_s\phi(\beta(s)))+c(H)ds\\
	&\leq \int^t_{s_1}L(\beta(s),\dot{\beta}(s),0)+c(H)ds\\
	&=h^{t-s_1}(\beta(s_1),x)+c(H)(t-s_1).
\end{align*}
If $t-s_1\geq\frac{1}{2}$, then $h^{t-s_1}(\beta(s_1),x)+c(H)(t-s_1)$ is bounded from above.
If not, then
$s_1>\frac{1}{2}$. Note that
$$
h^t(\gamma_{x,\epsilon}(0),x)=h^{s_1}(\gamma_{x,\epsilon}(0),\beta(s_1))+h^{t-s_1}(\beta(s_1),x).
$$
We derive that
$$
h^{t-s_1}(\beta(s_1),x)+c(H)(t-s_1)=\bigg(h^t(\gamma_{x,\epsilon}(0),x)+c(H)t\bigg)-\bigg(h^{s_1}(\gamma_{x,\epsilon}(0),\beta(s_1))+c(H)s_1\bigg),
$$

Note that the first term is bounded from above ($t\geq 1$) and the second term is bounded from below ($s_1> 1/2$), see Lemma 5.3.2 in \cite{F}.
Hence, $\{\cT^{\epsilon-}_t\phi(\cdot)|\epsilon\in(0,1],t\geq 1\}$ is uniformly bounded from above.
\endproof
\begin{lem}
The family $\{u_\eps^-\}$ is uniformly bounded for all $\eps\in(0,1]$.
\end{lem}
\begin{proof}
Due to the boundedness of  $\{\cT^{\epsilon-}_t\phi(\cdot)|\epsilon\in(0,1],t\geq 1\}$,   there exists a $K>0$ such that for $\phi\in C(M,\mathbb{R})$ satisfying $\|\phi\|\leq 1$, 
$$
|\cT^{\epsilon-}_t\phi(x)|\leq K, (x,t)\in M\times[1,+\infty)\mbox{ and } \epsilon\in (0,1].
$$

Since $\lim_{t\to+\infty}\mathcal{T}^{\epsilon-}_t\phi(x)=u_{\epsilon}^{-}(x)$ is uniquely established,
 we get $|u^-_\epsilon(x)|\leq K$ for all $\epsilon\in(0,1]$ immediately.
\end{proof}

\begin{lem}
	The map $x \rightarrow u_{\epsilon}^-(x)$ is equi-Lipschitz for $\epsilon\in(0,1]$.
\end{lem}
\proof
 Let $x,y\in M$ and $\alpha:[0,d_R(x,y)]\to M$ be a geodesic connecting $x$ and $y$.
Note that $\sqrt{\langle\dot{\alpha},\dot{\alpha}\rangle}_R\leq 1$. We derive that
\begin{align*}
u^-_{\epsilon}(y)-u^-_{\epsilon}(x)&\leq \int^{d_R(x,y)}_0L(\alpha(s),\dot{\alpha}(s),\epsilon u^-_{\epsilon}(\alpha(s)))+c(H)ds\\
&\leq \int^{d_R(x,y)}_0L(\alpha(s),\dot{\alpha}(s),0)+\Delta\cdot K+c(H)ds\\
&\leq (C_1+\Delta\cdot K+c(H))d_R(x,y),
\end{align*}
where $C_1$ is a uniform constant such that 
\[
L(x,v,0)\leq C_1,\quad\forall \|v\|_R\leq 1.
\]
By switching the role of $x$ and $y$, we derive
$$
u^-_{\epsilon}(x)-u^-_{\epsilon}(y)\leq (C_1+\Delta\cdot K+c(H))d_R(x,y)
$$
and then
$$
|u^-_{\epsilon}(x)-u^-_{\epsilon}(y)|\leq (C_1+\Delta\cdot K+c(H))d_R(x,y),
$$
So we finish the proof.
\endproof

\end{document}